\documentclass[12pt]{article}
\usepackage{graphicx}
\usepackage{caption}
\usepackage[T1]{fontenc}
\usepackage[cp1250]{inputenc}
\usepackage[top=15mm,bottom=15mm,left=23mm,right=23mm]{geometry}
\usepackage{amsmath,amstext,amssymb,amsopn,color,amsthm,textcomp}
\usepackage[dvipsnames]{xcolor}
\usepackage{eucal,mathrsfs,dsfont}
\usepackage{enumerate}
\usepackage[normalem]{ulem}
\usepackage{cancel}

\usepackage{tikz}
\usetikzlibrary{calc}
\usetikzlibrary{through,shapes.geometric}
\usepackage{pgfplots}
\pgfplotsset{width=9cm,compat=1.16}

\theoremstyle{plain}
\newtheorem{thm}{Theorem}[section]

\newtheorem{lemma}[thm]{Lemma}
\newtheorem{prop}[thm]{Proposition}

\newtheorem{problem}[thm]{Problem}

\theoremstyle{definition}

\theoremstyle{remark}
\newtheorem*{rem*}{Remark}
\newtheorem{rem}{Remark}[section]
\newcommand{\eps}{\varepsilon}

\newcommand{\R}{\mathbb{R}}

\newcommand{\N}{\mathbb{N}}

\renewcommand{\leq}{\leqslant}
\renewcommand{\geq}{\geqslant}

\DeclareMathOperator{\dist}{dist}

\DeclareMathOperator{\supp}{supp}
\newcommand{\Dn}{{D_{\Omega}^{1/2}}}

\def\({\left(}
\def\){\right)}
\def\[{\left[}
\def\]{\right]}
\def\<{\langle}
\def\>{\rangle}

\title{Qualitative properties of free boundaries for the exterior Bernoulli problem for the half Laplacian.}
\author{
\ Sven Jarohs\footnote{Goethe-Universit\"at Frankfurt, Germany, jarohs@math.uni-frankfurt.de.},
Tadeusz Kulczycki\footnote{Wroclaw University of Science and Technology, Poland, tadeusz.kulczycki@pwr.edu.pl},
Paolo Salani\footnote{DiMaI, Universit\`a di Firenze, Italy, paolo.salani@unifi.it}
}

\date{}

\begin{document}
\maketitle

\begin{abstract} 
In this work, we study the asymptotic behavior of the free boundary of the solution to the exterior Bernoulli problem for the half Laplacian when the Bernoulli's gradient parameter tends to $0^+$ and to $+\infty$. Moreover, we show that, under suitable conditions, the perpendicular rays of the free boundary always meet the convex envelope of the fixed boundary.
\end{abstract}

\noindent \textbf{Keywords:} free boundary problems, fractional Laplacian, moving planes method, starshapedness.

\noindent \textbf{AMS Subject Classification 2020:} 35R35, 35S99.

\section{Introduction}

The exterior Bernoulli free boundary problem for the half Laplacian is formulated as follows. 
\begin{problem}
\label{bernoulli_problem}
Given $d \in \N$, a bounded domain $K \subset \R^d$ and a constant $\lambda > 0$, we look for a continuous function $u: \R^d \to [0,1]$ and a domain $\Omega \supset \overline{K}$ of class $C^1$ satisfying
\begin{equation*}
\left\{
\begin{aligned}
(-\Delta)^{1/2}u(x)&=0 &&\text{for $x \in \Omega \setminus \overline{K}$,}\\
u(x)&=1 &&\text{for $x \in \overline{K}$,}\\
u(x)&=0 &&\text{for $x \in \Omega^c$,}\\
\Dn u(x) &= \lambda&& \text{for $x \in \partial \Omega$.}
\end{aligned}
\right.
\end{equation*}
\end{problem}
Here, $(-\Delta)^{1/2}$ denotes the half Laplacian given by
$$
(-\Delta)^{1/2}f(x) = \frac{\Gamma\left(\frac{d+1}{2}\right)}{2 \pi^{\frac{d+1}{2}}}\int_{\R^d}\frac{2 f(x) - f(x+z) - f(x-z)}{|z|^{d+1}} \, dz
$$
and $\Dn$ denotes the generalized normal derivative given by
$$
\Dn u(x) = \lim_{t \to 0^+} \frac{u(x + t n(x)) - u(x)}{t^{1/2}} = \lim_{t \to 0^+} \frac{u(x + t n(x))}{t^{1/2}},
$$
where $x\in \partial \Omega$ and $n(x)$ is the inward unit normal vector to $\Omega$ at $x$. As usual, by a domain we understand a nonempty, connected open set; by a domain of class $C^k$ for some $k\in \N$ we understand a domain whose boundary is locally a graph of a $C^k$ function. Given $r>0$ and $\xi\in\R^d$, we denote by $B_r(\xi)$ the open ball of radius $r$ centered at $\xi$. Since we will widely use the notion of starshapedness, although it is quite standard, let us recall it now: we say that a domain $A$ is {\em starshaped with respect to a point $\xi\in A$} if $\mu(A-\xi)+\xi\subseteq A$ for every $\mu\in[0,1]$ (i.e., if for every $x\in A$ the whole segment joining $\xi$ to $x$ is contained in $A$). Furthermore, we say that $A$ is starshaped with respect to the set $B\subseteq A$ if it is starhaped with respect to every point $\xi\in B$.

The study of classical Bernoulli free boundary problems has a long history, which started with the pioneering work by \cite{B1957}. The fractional version has been introduced in \cite{CRS2010}, where the authors studied the regularity of the free boundary. Since then, the study of the regularity of the free boundary in the nonlocal case has attracted great attention, see e.g. \cite{AR2020, AP2012, DSS2015, DSS2015b, SSS2014, DPTV23, DSV2015, EFY2023, EKPSS2019, FR2024, KW2023, RS2014, RW2024a, RW2024b, ST2024}. By assuming further geometric properties of the fixed boundary $\partial K$ in Problem \ref{bernoulli_problem}, these properties carry over to the solution $(u,\Omega)$. In \cite{JKS2022} (see Theorem 1.6 and Proposition 2.11) the following result is proven.

\begin{thm}[Theorem 1.6 and Proposition 2.11 in \cite{JKS2022}]
\label{Jarohs_Kulczycki_Salani}
Assume that the bounded domain $K \subset \R^d$ has a $C^2$ boundary and it is starshaped with respect to a ball $B_r(x_0)$ for some $r > 0$ and $x_0 \in K$. Then for any $\lambda > 0$ there exists a unique solution $u_{\lambda}$, $\Omega_{\lambda}$ of Problem \ref{bernoulli_problem}. Moreover, $\Omega_{\lambda}$ is bounded, starshaped with respect to $B_r(x_0)$ and it is of class $C^{\infty}$.
\end{thm}

The aim of this paper is to study geometric properties of the family of sets $\{\Omega_{\lambda}\}_{\lambda > 0}$.

For any $\lambda_1, \lambda_2 > 0$, we define
$$
\triangle_{x_0}(\Omega_{\lambda_1}, \Omega_{\lambda_2}) = \inf\{|\ln \mu|: \, \mu\leq 1,\,\mu (\Omega_{\lambda_1}-x_0) \subseteq \Omega_{\lambda_2}-x_0\,\text{ and }\,\mu (\Omega_{\lambda_2}-x_0) \subseteq \Omega_{\lambda_1}-x_0\}.
$$ 
We simply write $\triangle$ instead of $\triangle_{x_0}$ when $x_0=0$, as we can usually assume up to a translation.

We say that a domain {\em $A \subset \R^d$ satisfies the uniform interior ball condition (with radius $r>0$)} if for any point $x\in\partial A$ there exists a ball $B_r(\xi)\subseteq A$ such that $x\in\partial B_r(\xi)\cap\partial A$. We denote by $r_A >0$ the supremum of all $r > 0$ such that $A$ satisfies the uniform interior ball condition with radius $r>0$ (note that the supremum is in fact the maximum {if $A$ is bounded}). It is clear that if a bounded domain $A \subset \R^d$ has a $C^2$ boundary then it satisfies the uniform interior ball condition.
 We show the following results.

\begin{thm}
\label{family_D_lambda}
Assume that a bounded domain $K \subset \R^d$ has a $C^2$ boundary and it is starshaped with respect to a ball $B_\rho(x_0)\subset K$ for some $\rho > 0$. Then the following properties hold.
\begin{enumerate}[(i)]
\item $0 < \lambda_1 < \lambda_2$ implies 
$$
\Omega_{\lambda_1} \supset \Omega_{\lambda_2}.
$$
\item For any $\lambda_1, \lambda_2 > 0$ we have
$$
\triangle_{x_0}(\Omega_{\lambda_1}, \Omega_{\lambda_2}) \le 2|\ln \lambda_2 - \ln \lambda_1|.
$$
\item $$
\lim_{\lambda\to+\infty}\dist(\partial{\Omega_{\lambda}}, \partial K) = 0,\qquad
\lim_{\lambda\to0^+}\dist(\partial{\Omega_{\lambda}}, \partial K) = +\infty\,.
$$
\item\label{item 4 of main theorem1}  There is a strictly decreasing function $g_{d,r_K}$, depending only on $r_K$ and on the dimension $d$, such that for every $\lambda>0$ it holds
$$
g_{d,r_K}(\lambda)\le\dist(\partial {\Omega_{\lambda}}, \partial K) \le \frac{1}{\lambda^2}\,.
$$
The function $g_{d,r_K}$ satisfies $\lim_{\lambda\to0^+}g_{d,r_K}(\lambda)=+\infty$. We have
$$
g_{1,r_K}(\lambda)=\sqrt{\frac{4Cr_K}{\lambda^2}+r_K^2}-r_K\,,
$$
and
\begin{equation}\label{gdrk}
	g_{d,r_K}(\lambda)\ge \frac{C}{\lambda^2\left(\min\left\{A(\lambda),A(\lambda)^{1/2d}\right\}+1\right)^{2d-1}}\quad\text{for }d\geq 2\,,
\end{equation}
where $C\in(0,1]$ is a constant only depending on the dimension $d$ and
$$
A(\lambda)=\frac{C}{r_K\lambda^2}\,.
$$
\end{enumerate}
\end{thm}

\begin{rem}
We remark that if $K$ is convex with boundary of class $C^2$, then 
$$
r_K=\frac{1}{\kappa}\,,\quad\text{where $\kappa=\max\{\kappa_i(x)\,:\, i=1,\dots,d-1,\,x\in\partial K\}$}
$$
and $\kappa_1(x),\dots,\kappa_{d-1}(x)$ are the principal curvatures of $\partial K$ at $x$. This is true more in general, indeed for many regular domains, even not convex; on the other hand, without the convexity assumption it is easy to find smooth domains such that $r_K<1/\kappa$, for instance, this is the case for a dumbbell with a sufficiently tiny rod.
\end{rem}
\begin{rem}
	By  Theorem \ref{family_D_lambda} \eqref{item 4 of main theorem1}, we can see that there is a constant $D>0$, depending only on the dimension $d$, such that for every $\lambda>0$ it holds
	\begin{equation}
	\label{dist_Omega_K}
	\frac{D r_K^\beta}{\lambda^{2\alpha}}\le\dist(\partial {\Omega_{\lambda}}, \partial K) \le \frac{1}{\lambda^2}\,,
	\end{equation}
	where the exponents $\alpha$ and $\beta$ are as follows:
	$$
	\begin{array}{ll}
			\left\{\begin{array}{ll}\alpha=1\,,\,\,\beta=0\,&\text{for $\lambda$ large enough}\,,\\
			\\
			\alpha=1/2d\,,\,\,\beta=(2d-1)/2d\,&\text{ for $\lambda$ small enough}\,.
		\end{array}\right.
	\end{array}
	$$
Essentially, notice that if $\lambda$ is large then $\dist(\partial {\Omega_{\lambda}}, \partial K)$ behaves like $1/\lambda^2$, which is different from the classical case for which it behaves like $1/\lambda$, see \cite[Theorem 15]{FR1997}.
\end{rem}

\begin{thm}
\label{inward_normal_ray}
In the same assumptions about $K$ as in Theorem \ref{Jarohs_Kulczycki_Salani}, for any $x \in \partial \Omega_{\lambda}$ the inward normal ray to $\partial \Omega_{\lambda}$ meets the convex envelope of $\overline{K}$.
\end{thm}
\begin{rem}
In fact, the assumptions on $K$ in Theorem \ref{inward_normal_ray} (i.e., $C^2$ regularity and starshapedness with respect to a ball) only serve to assure the existence of a suitably regular solution $(\Omega_\lambda,u_\lambda)$, via Theorem \ref{Jarohs_Kulczycki_Salani}.
So, as soon as we have such a solution, the conclusion of the theorem still holds, even without such assumptions.
Moreover, we remark that in the planar case, $d=2$, meeting the convex envelope of $K$ and meeting $K$ are equivalent, hence the thesis can be written, more intriguingly, as follows: {\em for any $x \in \partial \Omega_{\lambda}$, the inward normal ray to $\partial \Omega_{\lambda}$ meets $K$}.
\end{rem}

We notice that Theorem \ref{inward_normal_ray} gives interesting information about the shape of $\Omega_\lambda$, which, from a purely qualitative point of view, in particular says that, when $\lambda\to 0^+$, $\Omega_\lambda$ tends to look like a ball with center in $K$. Information about the shape of $\Omega_\lambda$ can be obtained also from \cite{JKS2019} and \cite{FJ2015,SV19}. For instance, when $K$ is convex, \cite{JKS2019} tells that $\Omega_\lambda$ must be starshaped with respect to every point of $K$, while \cite{FJ2015,SV19} informs that, if $K$ is symmetric with respect to a hyperplane, the same happens for $\Omega_\lambda$ -- hence if $K$ is a ball, then $\Omega_\lambda$ is a ball as well.

For the classical exterior Bernoulli free boundary problem properties of parametrized families of free boundaries have been studied in \cite{A1989, AM1995, FR1997, S1994, T1975}. In particular, in the classical case estimates of the type (\ref{dist_Omega_K}) have been proved in \cite[Theorem 15]{FR1997} and properties similar to the ones presented in Theorem \ref{family_D_lambda} (i), (ii) have been proved in \cite[Theorem 3.9 (ii), (iii)]{AM1995}. In \cite{S1994} it was proved that the normal to the free boundary always hits the convex hull of the fixed boundary (this was earlier shown for dimension $d = 2$ in \cite{T1975} by different methods).

The rest of the paper is organized as follows.
In Sections 2, we give some preliminaries and prove some technical lemmas. In Section 3, we prove the main results, that are Theorem \ref{family_D_lambda} and Theorem \ref{inward_normal_ray}.


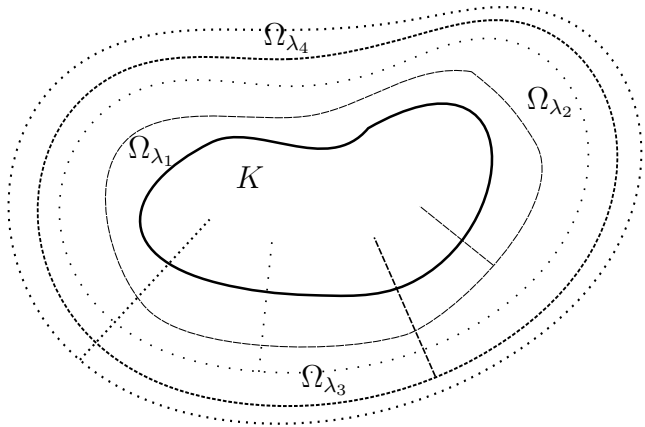
\begin{figure}
\centering
\begin{tikzpicture}[x=1cm,y=1cm,step=1cm]
\node at (0,0) {\includegraphics[width=0.5\textwidth]{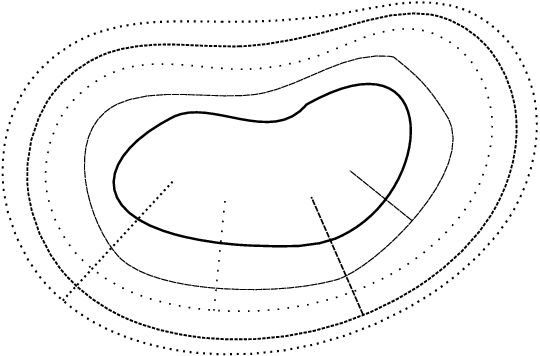}};
\node at (-1,0.5) {$K$};
\node at (-2.3,0.85) {$\Omega_{\lambda_1}$};
\node at (3,1.5) {$\Omega_{\lambda_2}$};
\node at (0,-2.2) {$\Omega_{\lambda_3}$};
\node at (-0.5,2.35) {$\Omega_{\lambda_4}$};
\end{tikzpicture}
\caption{$\Omega_{\lambda_i}$ for $i=1,2,3,4$ with $\lambda_1<\lambda_2<\lambda_3<\lambda_4$ and the corresponding normal rays intersecting $K$.}
\label{fig for theorem 2}
\end{figure}

%
%

\section{Preliminaries}\label{preliminearies}

Given a fixed open set $\Omega\subset \R^d$, we let $\delta(x)$ denote the distance of $x$ to $\R^d\setminus \Omega$. We use the following notation for halfspaces and reflections across the boundary of halfspaces. Given a halfspace $H\subset \R^d$, that is, $H=H_{\lambda,e}:=\{x\in \R^d\;:\; x\cdot e> \lambda\}$ for some $\lambda\in\R$ and $e\in \partial B_1(0)$, let $Q:=Q_{\lambda,e}:\R^d\to\R^d$ be the reflection across $\partial H$, that is
$$
\overline{x}:=Q(x)=x-2(x\cdot e)e+2\lambda e\quad\text{for $x\in \R^d$.}
$$
Moreover, for a function $u:\R^d\to\R$ let $\overline{u}:=u\circ Q$. By rotation, we may usually simply consider $e=e_1:=(1,0,\ldots,0)$.

Let $\Omega \subset \R^d$ be an open, bounded set. We
say that $u\in H^{1/2}(\R^d)$ satisfies (in weak sense)
	$$
	(-\Delta)^{1/2}u\geq 0\quad\text{in $\Omega$,}
	$$
if for all nonnegative $v\in H^{1/2}(\R^d)$ with $\supp\,v\subset \Omega$ it holds
$$
\int_{\R^d}\int_{\R^d}\frac{(u(x)-u(y))(v(x)-v(y))}{|x-y|^{d+1}}\,dxdy\geq 0.
$$
Here, as usual, $H^{1/2}(\R^d)$ denotes the (fractional) Sobolev space of order $\frac12$.\\
We emphasize that the solution given by Theorem \ref{Jarohs_Kulczycki_Salani} belongs\footnote{It is noteworthy here, that indeed neither $H^{1/2}(\R^d)$ nor $C^{1/2}_c(\R^d)$ are subsets of each other. Thus we study functions belonging to the intersection of these two spaces.} to $H^{1/2}(\R^d)$ by construction. {{The following Lemma implies, that it also belongs to $C^{1/2}(\R^d)$.}} 

\begin{lemma}
\label{C12}
{{Let $\Omega \subset \R^d$ be an open, nonempty, connected, bounded set, which has a $C^{1,1}$ boundary, and let $K \subset \Omega$ be open and nonempty with $C^{1,1}$ boundary, such that $\dist(K, \partial(\Omega)) > 0$. Assume that $u: \R^d \to [0,1]$, $u \in H^{1/2}(\R^d)$ and satisfies $u = 1$ on $\overline{K}$, $u = 0$ on $\Omega^c$ and $(-\Delta)^{1/2} u = 0$ in $U := \Omega \setminus \overline{K}$. Then $D_{U}^{1/2}u(x)$ is well defined for all $x \in \partial U$  { and $u\in C^{1/2}(\R^d)$.}}}
\end{lemma}
\begin{proof} {{Denote $r = \dist(K, \partial(\Omega))$. Let $v \in C_c^{\infty}(\R^d)$ be such that $v \equiv 1$ on $\{x \in \R^d: \, \dist(x,K) < r/2\}$ and $\supp(v) \subset \Omega$. Put $w = u - v$. Note that $(-\Delta)^{1/2} v$ is bounded, denote $f = (-\Delta)^{1/2} v$. Then $w$ satisfies
\begin{equation*}
\left\{
\begin{aligned}
(-\Delta)^{1/2}w(x)&=-f(x), &&\text{for $x \in \Omega \setminus \overline{K}$,}\\
w(x)&=0, &&\text{for $x \in \overline{K} \cup \Omega^c$.}
\end{aligned}
\right.
\end{equation*}
Note that $U$ is a bounded domain with $C^{1,1}$ boundary. Hence, by Proposition 1.1 and Theorem 1.2 in \cite{RS2014}, we obtain $w \in C^{1/2}(\R^d)$ and the function $x\mapsto w(x)\dist(x,\R^d\setminus U)^{-1/2}$ belongs to $C^{\alpha}(\overline{U})$ for some $\alpha\in(0,\frac12)$. Hence $D_{U}^{1/2}u(x) = D_{U}^{1/2}w(x)$ is well defined for all $x \in \partial U$. Clearly, $v \in C^{1/2}(\R^d)$. Hence, $u \in C^{1/2}(\R^d)$.}}
\end{proof}

\begin{prop}[Fractional Hopf lemma]\label{prop:hopf}
Let $U\subset \R^d$ be an open bounded set and let $u\in H^{\frac12}(\R^d) \cap C^{1/2}(\R^d)$ satisfy
$$
(-\Delta)^{1/2}u\geq 0\quad\text{in $U$;}\quad u\geq 0\quad\text{in $\R^d\setminus U$.}
$$
Then either $u\equiv 0$ in $\R^d$ or $u>0$ in $U$.\\
Moreover, if $u>0$ in $U$ and, in addition, there is $x_0\in \partial U$ such that $u(x_0)=0$ and there is a ball $B\subset U$ with $\partial B\cap \partial U=\{x_0\}$, then 
{{
$$
\liminf_{t\to 0^+}\frac{u(x_0+tn(x_0))}{t^{1/2}}>0,
$$
where $n(x_0)$ is the inward unit normal vector to $U$ at $x_0$. In particular, if $D_{U}^{1/2} u(x_0)$ exists, then it is positive.}}
\end{prop}
\begin{proof}
This statement follows from \cite{FJ2015}, in particular combining Proposition 3.3, Remark 3.5 and Proposition 3.1 therein.
\end{proof}

\begin{prop}[Fractional Hopf lemma -- a variant for antisymmetric functions]\label{prop:hopf2}
Let $H\subset \R^d$ be a halfspace and let $Q:\R^d\to\R^d$ be the reflection at $\partial H$. Let $W\subset H$ open and let $v\in H^{\frac12}(\R^d) \cap C^{1/2}(\R^d)$ such that $v(Q(x))=-v(x)$ for all $x \in H$. If
$$
(-\Delta)^{1/2}v\geq 0\quad\text{in $W$;}\quad v\geq 0\quad\text{in $H\setminus W$.}
$$
Then either $v\equiv 0$ in $\R^d$ or $v>0$ in $W$.\\
Moreover, if $v>0$ in $W$ and, in addition, there is $x_0\in \partial W\setminus \partial H$ such that $v(x_0)=0$ and there is a ball $B\subset W$ with $\partial B\cap \partial W=\{x_0\}$, then 
{{
$$
\liminf_{t\to 0^+}\frac{v(x_0+tn(x_0))}{t^{1/2}}>0,
$$
where $n(x_0)$ is the inward unit normal vector to $W$ at $x_0$. In particular, if $D_{W}^{1/2} v(x_0)$ exists, then it is positive.}}
\end{prop}
\begin{proof}
	This is a direct consequence of Proposition 3.1 and Proposition 3.3 in \cite{FJ2015}.
\end{proof}

	\begin{lemma}[A fractional corner point lemma, Lemma 4.4, \cite{FJ2015}]\label{cornerpoint}
Let $H=\{x\in \R^d\;:\; x_1>0\}$ and let $Q:\R^d \to \R^d$ be the reflection at $\partial H$. Let $W\subset \R^d$ be an open set, which is symmetric in $x_1$, that is $Q(W)=W$, and such that $0\in \partial W$. Assume further that the interior normal of $W$ at $0$ is given by $e_2=(0,1,0,\ldots,0)$. Let $U:=W\cap H$ (cf. Figure \ref{fig for corner point lemma}). Let $v\in H^{\frac12}(\R^d) \cap C^{1/2}(\R^d)$ such that $v(Q(x))=-v(x)$ for all $x \in H$ and
		$$
		(-\Delta)^{1/2}v\geq 0\quad\text{in $U$;}\quad v\geq 0\quad\text{in $H\setminus U$,}\quad v>0 \quad\text{in $U$.}
		$$
		Let $\eta=(1,1,0,\ldots,0)$. Then there is $C,t_0>0$ depending only on $W$ and $d$ such that
		$$
		v(t\eta)\geq Ct^{\frac{3}{2}}\quad\text{for all $t\in(0,t_0)$.}
		$$
	\end{lemma}

\begin{figure}
\begin{center}
\begin{tikzpicture}[x=1cm,y=1cm,step=1cm]
\begin{axis}[axis line style={draw=none},tick style={draw=none},xticklabel=\empty,yticklabel=\empty]
		\addplot [smooth] coordinates {(0,4) (4,3) (5,0) (4,-3) (1,-3.87) (0.49,-4)};
		\addplot [smooth,thick, dotted] coordinates {(0,4) (-4,3) (-5,0) (-4,-3) (-1,-3.87) (-0.49,-4)};
		\addplot [sharp plot,thick, dotted] coordinates { (-0.5,-4) (0,-4)  };
		\addplot [sharp plot] coordinates {  (0,-4) (0.5,-4) };
			\addplot [sharp plot,dashed] coordinates { (0,6) (0,-4) (0,-6) };
			\filldraw (0,-4) circle (1pt);
			\node at (-0.6,-4.6) {$0$};
			\node at (0.6,5) {$H$};
			\node at (1.6,2) {$U$};
			\draw[->,thick] (0,-4)--(0,-1);
			\draw[->,thick] (0,-4)--(2.55,-4);
			\draw[->,thick] (0,-4)--(2.4,-1.05);
			\node at (-0.6,-2.15) {$e_2$};
			\node at (1.4,-4.7) {$e_1$};
			\node at (0.7,-2.25) {$\eta$};
\end{axis}
\end{tikzpicture}
\end{center}
\caption{Exemplification of Lemma \ref{cornerpoint} with $U=W\cap H$ and respectively of Lemma \ref{normal_derivative} with $U=\Omega_+$ (but without $K$), and $e_1=(1,0,\ldots,0)$ and $e_2=(0,1,0,\ldots,0)$.}
\label{fig for corner point lemma}
\end{figure}
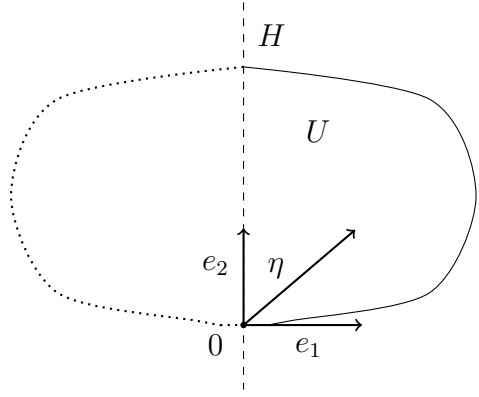

\begin{lemma}\label{scaling}
Let $K\subset \R^d$ be a bounded domain, which has $C^2$ boundary and it is starshaped with respect to a ball $B_r(x_0)$ for some $r > 0$ and $x_0 \in K$. Let $\lambda>0$ and let $(u_{\lambda},\Omega_{\lambda})$ be the unique solution of Problem \ref{bernoulli_problem} given by \cite[Theorem 1.6]{JKS2022}. For $t>0$ let $v: \R^d \to \R^d$ be given by $v(x)=u_{\lambda}(x/t)$ for $x\in \R^d$. Then $v$ is the unique solution of
\begin{equation*}
\left\{
\begin{aligned}
(-\Delta)^{1/2}v(x)&=0 &&\text{for $x \in t \Omega_{\lambda} \setminus t\overline{K}$,}\\
v(x)&=1 &&\text{for $x \in t\overline{K}$,}\\
v(x)&=0 &&\text{for $x \in (t\Omega_{\lambda})^c$,}\\
D_{t \Omega_{\lambda}}^{1/2} v(x) &= t^{-\frac12}\lambda&& \text{for $x \in \partial (t\Omega_{\lambda})$.}
\end{aligned}
\right.
\end{equation*}
That is, $v$ is the unique solution of Problem \ref{bernoulli_problem} with the respective scaled quantities. 
\end{lemma}
\begin{proof}
This follows immediately by the uniqueness statement in \cite[Theorem 1.6]{JKS2022} with the scaling properties of $(-\Delta)^{1/2}$.
\end{proof}

\begin{lemma}\label{boundary calc}
Let $0<r<R$, $x_0\in \R^d$, and let $b\in H^{1/2}(\R^d)\cap C^{1/2}(\R^d)$ be the solution to
$$
(-\Delta)^{\frac12}b=0\quad\text{in $W$},\quad b=1\quad \text{in $\overline{B_r(x_0)}$, and}\quad b=0\quad\text{in $B_R^c(x_0)$,}
$$
where $W = B_R(x_0)\setminus \overline{B_r(x_0)}$. Then $b$ is radially symmetric, strictly decreasing in the radial direction away from $x_0$, and there is a constant $C_d\in(0,1)$ depending only on $d$ such that
$$
\frac{C_d}{\sqrt{R-r}}\big(\frac{r}{R}\big)^{d-\frac12}\leq D_{W}^{1/2} b(\theta)< \frac{1}{\sqrt{R-r}}\quad\text{for any $\theta\in \partial B_R(x_0)$.}
$$
\end{lemma}
\begin{proof}
First note that the symmetry and monotonicity properties of $b$ follow immediately from \cite[Theorem 1.2]{SV19}. {{The fact that $D_{W}^{1/2} b(\theta)$ is well defined for any $\theta\in \partial B_R(x_0)$ follows from Lemma \ref{C12}.}} In the following, we can assume $x_0=0$, without loss of generality, and consider only the normal derivative in the direction $e_1=(1,0,\ldots,0)\in \R^d$. That is, we set $\theta=Re_1$. For the upper bound, we consider the half space $H=\{x\in \R^d\;:\; x_1<R\}$. We abbreviate $B_R(0)$ to $B_R$ and $B_r(0)$ to $B_r$. Note that $B_{R}\subset H$ and $\partial H\cap \partial B_R=\{Re_1\}$. Moreover, recall the function
$$
v:\R^d\to\R, \quad v(x)=\left\{\begin{aligned} &0, && x\in \R^d\setminus H;\\
&(R-x_1)^{\frac{1}{2}}, && x\in H
\end{aligned}\right.
$$
which satisfies
$$
(-\Delta)^{\frac12}v=0\quad\text{in $H$,}\quad v=0 \quad\text{in $\R^d\setminus H$.}
$$
A simple observation gives
$$
\inf_{\substack{x\in \R^d\\ x\cdot e_1<r}} v(x)=v(re_1)=(R-r)^{\frac12}.
$$
Thus the function $\tilde{v}:\R^d\to\R$, $\tilde{v}(x)=(R-r)^{-\frac12}v(x)$ satisfies
$$
(-\Delta)^{\frac12}\tilde{v}=0\quad\text{in $W$},\quad \tilde{v}\geq 0 \quad\text{in $\R^d\setminus B_R$, and}\quad \tilde{v}\geq 1\quad\text{in $B_r$.}
$$
The fractional Hopf Lemma applied to $\tilde{v}-b$ implies $D_{W}^{1/2} (\tilde{v}-b)(Re_1)>0$ and thus
$$
D_{W}^{1/2} b(Re_1)< D_{W}^{1/2} \tilde{v}(Re_1)=\frac{1}{\sqrt{R-r}}\lim_{t\to0}\frac{(R-((R-t))^{\frac12}}{\sqrt{t}}=\frac{1}{\sqrt{R-r}}.
$$
Next, let $y_0=\frac{R+r}{2}e_1$ and let $\rho=\frac{R-r}{2}$, so that $B_{\rho}(y_0)\subset B_{R}\setminus B_r$ and $\partial B_{\rho}(y_0)\cap \partial B_{R}=\{Re_1\}$, cf. Figure \ref{fig for concentric balls}. Recall that the Poisson kernel of $(-\Delta)^{\frac12}$ in $B_{\rho}(y_0)$ is given by
\begin{figure}
\begin{center}
\begin{tikzpicture}[x=1cm,y=1cm,step=1cm]
\draw [line width=0.3mm]  (11,3) circle (30pt);
\draw [line width=0.3mm]  (11,3) circle (90pt);
\draw [dotted, line width=0.3mm] (13.108,3) circle (30pt);
\draw [dotted, line width=0.2mm] (11,3) -- (14.17,3);
\filldraw (11,3) circle (1pt);
\node at (11.3,2.8) {$0$};
\filldraw (12.05,3) circle (1pt);
\node at (12.35,2.8) {$re_1$};
\filldraw (13.108,3) circle (1pt);
\node at (13.408,2.8) {$y_0$};
\filldraw (14.17,3) circle (1pt);
\node at (14.5,2.8) {$Re_1$};
\node at (13.4,3.6) {$B_{\rho}(y_0)$};
\end{tikzpicture}
\end{center}
\caption{Definition of $B_{\rho}(y_0)$ with $x_0=0$.}
\label{fig for concentric balls}
\end{figure}
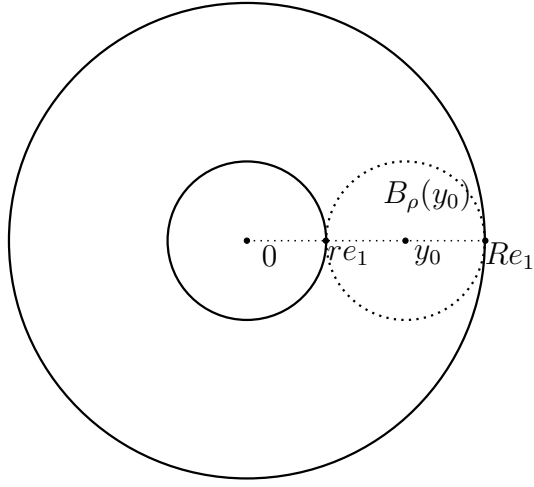

\begin{align*}
P(x,y)=c\frac{(\rho^2-|x-y_0|^2)^{\frac12}}{(|y-y_0|^2-\rho^2)^{\frac12}}|x-y|^{-d},\quad x\in B_{\rho}(y_0),\ y\in \R^d\setminus \overline{B_{\rho}(y_0)}
\end{align*}
with $c= \Gamma(\frac{d}{2})  \pi^{-1-\frac{d}{2}}$. Thus we have for $x\in B_{\rho}(y_0)$
$$
b(x)=\int_{\R^d\setminus B_{\rho}(y_0)}P(x,y)b(y)\,dy
$$
and 
\begin{align}
D_{W}^{1/2} b(Re_1)&=\lim_{t\to 0^+}\frac{b((R-t)e_1)}{\sqrt{t}}=\lim_{t\to 0^+}t^{-\frac12}\int_{\R^d\setminus B_{\rho}(y_0)}P((R-t)e_1,y)b(y)\,dy\notag\\
&=c\lim_{t\to0^+}\int_{\R^d\setminus B_{\rho}(y_0)}\frac{(\rho^2-|(R-t)e_1-y_0|^2)^{\frac12}}{\sqrt{t}(|y-y_0|^2-\rho^2)^{\frac12}}\frac{b(y)}{|(R-t)e_1-y|^d}\,dy\notag\\
&>c\lim_{t\to0^+}\int_{B_r}\frac{(\rho^2- (\rho -t)^2)^{\frac12}}{\sqrt{t}(|y-y_0|^2-\rho^2)^{\frac12}}\frac{1}{|(R-t)e_1-y|^d}\,dy\notag\\
&=c\lim_{t\to0^+}\int_{B_r((R-t)e_1)}\frac{(2\rho -t)^{\frac12}}{(|(\rho-t)e_1-z|^2-\rho^2)^{\frac12}}\frac{1}{|z|^d}\,dz\notag\\
&=c \int_{B_r(Re_1)}\frac{(2\rho)^{\frac12}}{(|z|^2-2\rho z_1)^{\frac12}}\frac{1}{|z|^d}\,dz.\label{form1}
\end{align}

If $d=1$, it follows that 
\begin{align*}
D_{W}^{1/2} b(Re_1)&>c \int_{R-r}^{R+r}\frac{(2\rho)^{\frac12}}{(z^2-2\rho z)^{\frac12}}\frac{1}{z}\,dz=\frac{1}{\pi} \int_{\frac{R-r}{2\rho}}^{\frac{R+r}{2\rho}}\frac{(2\rho)^{\frac12}}{(4\rho^2\tau^2-4\rho^2 \tau)^{\frac12}}\frac{2\rho}{2\rho \tau}\,d\tau\\
&=\frac{1}{\pi \sqrt{2\rho}} \int_{1}^{\frac{R+r}{R-r}}\frac{1}{\tau (\tau^2-\tau)^{\frac12}}\,d\tau=\frac{1}{\pi \sqrt{2\rho}} \frac{2\sqrt{\tau-1}}{\sqrt{\tau}}\Bigg|_{1}^{\frac{R+r}{R-r}}\\
&=\frac{\sqrt{2}}{\pi \sqrt{\rho}}\cdot \frac{\sqrt{R+r-(R-r)}}{\sqrt{R+r}}=\frac{2}{\pi\rho^{\frac12}}\cdot \frac{\sqrt{r}}{\sqrt{R+r}}\geq\frac{\sqrt{2}}{\pi \sqrt{\rho}}\sqrt{\frac{r}{R}}
\end{align*}
and the claim follows for $d=1$. If $d\geq 2$, first note that 
$$
(x_1-R)^2+|x'|^2\leq r^2\quad\Leftrightarrow\quad |x'|^2\leq r^2-(x_1-R)^2=(x_1-2\rho)(r+R-x_1).
$$
Since, for $x_1\in[0,R]$ we have
$$
(x_1-2\rho)\leq r+R-x_1,
$$
using $(d-1)$ dimensional spherical coordinates, from \eqref{form1} it follows that
\begin{align}
D_{W}^{1/2} b(Re_1)&>c \int_{B_r(Re_1)}\frac{(2\rho)^{\frac12}}{(z_1^2-2\rho z_1+|z'|^2)^{\frac12}}\frac{1}{(z_1^2+|z'|^2)^{\frac{d}{2}}}\,dz\notag\\
&\geq c \int_{2\rho}^{R}\int_{B_{\sqrt{(\tau-2\rho)(r+R-\tau)}}}\frac{(2\rho)^{\frac12}}{(\tau^2-2\rho \tau+|z'|^2)^{\frac12}}\frac{1}{(\tau^2+|z'|^2)^{\frac{d}{2}}}dz'\,d\tau \notag\\
&\geq c \int_{2\rho}^{R}\int_{B_{\tau-2\rho}}\frac{(2\rho)^{\frac12}}{(\tau^2-2\rho \tau+|z'|^2)^{\frac12}}\frac{1}{(\tau^2+|z'|^2)^{\frac{d}{2}}}dz'\,d\tau \notag\\
&=\frac{\Gamma(\frac{d}{2})2\pi^{\frac{d-1}{2}}}{\Gamma(\frac{d-1}{2})\pi^{1+\frac{d}{2}}} \int_{2\rho}^{R}\int_{0}^{\tau-2\rho}\frac{(2\rho)^{\frac12}}{(\tau^2-2\rho \tau+t^2)^{\frac12}}\frac{t^{d-2}}{(\tau^2+t^2)^{\frac{d}{2}}}dt\,d\tau =:A\label{d geq 2 case step1}\notag
\end{align}
Substituting $t=\tau x$ and then $\tau=2\rho y$ we get
\begin{align*}
A
&=\frac{\Gamma(\frac{d}{2})2(2\rho)^{-\frac12}}{\Gamma(\frac{d-1}{2})\pi^{\frac{3}{2}}} \int_{1}^{\frac{R}{2\rho}}\int_{0}^{1-\frac{1}{y}}\frac{1}{y(y^2(1+x^2)-y)^{\frac12}}\frac{x^{d-2}}{(1+x^2)^{\frac{d}{2}}}dx\,dy\notag\\
&\geq \frac{\Gamma(\frac{d}{2})2(2\rho)^{-\frac12}}{\Gamma(\frac{d-1}{2})\pi^{\frac{3}{2}}} \int_{1}^{\frac{R}{2\rho}}\frac{1}{y\Big(y^2(1+(1-\frac{1}{y})^2)-y\Big)^{\frac12}}\int_{0}^{1-\frac{1}{y}} \frac{x^{d-2}}{(1+x^2)^{\frac{d}{2}}}\,dx\,dy\notag\\
&\geq \frac{\Gamma(\frac{d}{2})2(2\rho)^{-\frac12}}{2^{\frac{d+1}{2}}\Gamma(\frac{d-1}{2})\pi^{\frac{3}{2}}} \int_{1}^{\frac{R}{2\rho}}\frac{1}{y^{\frac32}\sqrt{y-1}}\int_{0}^{\frac{y-1}{y}} x^{d-2}\,dx\,dy\notag\\
&=\frac{\Gamma(\frac{d}{2})(2\rho)^{-\frac12}}{2^{\frac{d+1}{2}}\Gamma(\frac{d+1}{2})\pi^{\frac{3}{2}}} \int_{1}^{\frac{R}{2\rho}}(y-1)^{d-\frac{3}{2}}y^{-\frac12-d}\,dy=\frac{\Gamma(\frac{d}{2})(2\rho)^{-\frac12}}{2^{\frac{d+1}{2}}\Gamma(\frac{d+1}{2})\pi^{\frac{3}{2}}}\Bigg(\frac{2}{2d-1}\Big(\frac{y}{y-1}\Big)^{\frac12-d}\Bigg)\Bigg|_{1}^{\frac{R}{2\rho}}\notag\\
&=\frac{\Gamma(\frac{d}{2}) }{2^{\frac{d-1}{2}}(2d-1)\Gamma(\frac{d+1}{2})\pi^{\frac{3}{2}}\sqrt{R-r}}\Big(\frac{r}{R}\Big)^{d-\frac12}.
\end{align*}
The claim thus also holds for $d\geq 2$.
\end{proof}

\begin{rem}
From the proof it follows that in the case of $d=1$ (and with $x_0=0$), the upper bound can easily be improved. With the notation in the proof it holds for $d=1$:
\begin{align*}
D_{W}^{1/2} b(Re_1)&=D_{W}^{1/2} b(R)=\lim_{t\to 0^+}\frac{b((R-t))}{\sqrt{t}}=\lim_{t\to 0^+}t^{-\frac12}\int_{\R\setminus (y_0-\rho,y_0+\rho)}P((R-t)e_1,y)b(y)\,dy\notag\\
&=\frac{1}{\pi}\lim_{t\to0^+}\int^{\frac{R+r}{2}-\rho}_{-R}\frac{(\rho^2-((R-t)-\frac{R+r}{2})^2)^{\frac12}}{\sqrt{t}((y-\frac{R+r}{2})^2-\rho^2)^{\frac12}}\frac{b(y)}{|(R-t)-y|}\,dy\notag\\
&=\frac{\sqrt{2\rho}}{\pi}\int_{\rho-\frac{R+r}{2}+R}^{2R}\frac{b(R-y)}{|y| ((\frac{R+r}{2}-R+y)^2-\rho^2)^{\frac12}}\,dy\notag\\
&=\frac{\sqrt{2\rho}}{\pi}\int_{2\rho}^{2R}\frac{b(R-y)}{|y| (y-\rho)^2-\rho^2)^{\frac12}}\,dy=
\frac{\sqrt{2\rho}}{\pi}\int_{2\rho}^{2R}\frac{b(R-y)}{y(y^2-2\rho y)^{\frac12}}\,d y \notag\\
&=\frac{\sqrt{2\rho}}{\pi}\int_{1}^{R/\rho}\frac{b(R-2\rho t)2\rho}{2\rho t (4\rho^2t^2-4\rho^2 t)^{\frac12}}\,dt=\frac{1}{\pi\sqrt{2\rho}}\int_{1}^{R/\rho}\frac{b(R-2\rho t)}{t ( t^2-t)^{\frac12}}\,dt.
\end{align*}
Since $b \le 1$ we get
\begin{align*}
D_{W}^{1/2} b(Re_1)&\leq \frac{1}{\pi\sqrt{2\rho}}\int_{1}^{\frac{R}{\rho}}\frac{1}{t^{\frac32} ( t-1)^{\frac12}}\,dt=\frac{\sqrt{2}}{\pi\sqrt{\rho}}\frac{\sqrt{t-1}}{\sqrt{t}}\Bigg|_{1}^{\frac{R}{\rho}} =\frac{\sqrt{2R-2\rho}}{\pi\sqrt{\rho R}}=\frac{\sqrt{R+r}}{\pi\sqrt{\rho  R}}.
\end{align*}
Thus, for $d=1$ it holds
\begin{equation}\label{precise N=1}
\frac{2}{\pi \sqrt{R-r}} \sqrt{\frac{r}{R}}\leq D_W^{1/2} b(\pm R)=\lim_{t\to 0^+}\frac{b(\pm R\mp t)}{\sqrt{t}}\leq \frac{\sqrt{2}}{\pi\sqrt{R-r}}\sqrt{1+\frac{r}{R}}.
\end{equation}
\end{rem}

\section{Proofs of the main results}

\begin{proof}[Proof of Theorem \ref{family_D_lambda}]
We assume, without loss of generality, that $x_0=0$.  {Let $u_i:=u_{\lambda_i}$, $\Omega_i:=\Omega_{\lambda_i}$, for $i=1,2$,  be the solution to Problem \ref{bernoulli_problem} with $\lambda=\lambda_i$ in $K$.} In the following, for $\tau>0$ let $v_\tau,w_{\tau}:\R^d\to\R$ be given by $v_\tau(x)=u_2(x/\tau)$ and $w_{\tau}(x)=u_1(x/\tau)$.\\
(i): Let $t>0$ be the largest number such that $(t\Omega_2)\subset \Omega_1$. Notice that for this $t$ we can find some $\theta\in \partial(t\Omega_2)\cap \partial \Omega_1$. Assume $t<1$. Then by the regularity of the free boundary $\Omega_1$ and $\Omega_2$ respectively, we can find a ball $B$ contained in $t\Omega_2\setminus \overline{K}$ such that $\partial B\cap \partial (t\Omega_2)=\{\theta\}$. Since $tK\subset K$, by the fractional Hopf Lemma, Proposition \ref{prop:hopf}, applied to $u_1-v_t$ in $t\Omega_2\setminus \overline{K}$ it follows that either we have $u_1-v_t\equiv 0$ or $D_{t\Omega_2}^{1/2} (u_1-v_t)(\theta)=\lambda_1-t^{-\frac12}\lambda_2>0$. Notice that since $t\in(0,1)$ by assumption and $\lambda_2>\lambda_1$, we must have $u_1\equiv v_t$. But by the maximum principle we also have $v_t<1=u_1$ in $K\setminus (t \overline{K})$. A contradiction. Thus we must have $t\geq 1$ and (i) follows.\\
(ii): Assume, without loss of generality, $\lambda_2>\lambda_1$ and let $s$ be the largest number such that $s\Omega_1\subset \Omega_2$ and let $\theta\in \partial(s\Omega_1)\cap \partial \Omega_2$. By (i) we have $s\in(0,1]$ and thus $\triangle(\Omega_1,\Omega_2)=|\ln s|$.
To estimate $s$, note that analogously to the proof of (i) we have by the fractional Hopf Lemma, applied to $u_2-w_s$ in $s\Omega_1\setminus \overline{K}$, that either $u_2- w_s\equiv 0$ or $D_{s\Omega_1}^{1/2} (u_2-w_s)(\theta)=\lambda_2-s^{-\frac12}\lambda_1>0$. Since $u_2\equiv w_s$ is not possible, because $s<1$ and thus $sK\subsetneq K$, it holds $s^{\frac12}\lambda_2>\lambda_1$, hence, in particular,
$$
0\geq\ln(s)\geq \ln(\lambda_1^2/\lambda_2^2)=2\Big(\ln(\lambda_1)-\ln(\lambda_2)\Big),
$$
that is,
$$
\triangle(\Omega_1,\Omega_2)\leq 2\big|\ln(\lambda_1)-\ln(\lambda_2)\big|
$$
as claimed in (ii).\\
(iii): Next, let $\xi_0\in\partial K$ and $\theta\in\partial \Omega_\lambda$ such that $\dist(\xi_0,\theta)=\dist(\partial \Omega_\lambda,\partial K)$. By the interior ball property, there exists a ball $B_{r_K}(x_0)$ of radius $r_K$ (we denote its center as $x_0$) such that $B_{r_K}(x_0)\subset K$ and $\xi_0\in\partial B_{r_K}(x_0)\cap\partial K$.  Let $B_R(x_0)$ be the ball with center $x_0$ and radius $R=r_K+\dist(\partial \Omega_\lambda,\partial K)$.  Notice that $B_R(x_0)\subset\Omega_\lambda$ and $\theta\in \partial \Omega_{\lambda}\cap \partial B_R(x_0)$. Let $b=b_{\lambda}$ be the solution to 
$$
(-\Delta)^{1/2}b=0\quad\text{in $W$,}\quad b=1\quad\text{in $\overline{B_{r_K}(x_0)}$, and}\quad b=0 \quad\text{in $B_R^c(x_0)$}
$$
where $W = B_R(x_0)\setminus \overline{B_{r_K}(x_0)}$. Then the comparison principle yields $b\leq u_{\lambda}$ in $\R^d$. Since $b(\theta)=u_{\lambda}(\theta)=0$, we have 
$$
D_{W}^{1/2} b(\theta)\leq D_{\Omega_{\lambda}}^{1/2} u_{\lambda}(\theta) = \lambda.
$$
Moreover, Lemma \ref{boundary calc} guarantees the existence of some constant $C>0$ depending on $d$ such that
$$
\frac{C}{\sqrt{\dist(\partial \Omega_{\lambda},\partial K)}}\left(\frac{r_K}{R}\right)^{d-\frac12}\leq D_{W}^{1/2} b(\theta)\,.
$$
Similarly, let $m:=\inf\{\rho>0\;:\; K\subset B_{\rho}(x_0)\}$, $M:=\inf\{\sigma>0\::\; \Omega_{\lambda}\subset B_\sigma(x_0)\}$, and $\beta=\beta_{\lambda}$ the unique solution to
$$
(-\Delta)^{1/2}\beta=0\quad\text{in $U$,}\quad b=1\quad\text{in $\overline{B_m(x_0)}$, and}\quad b=0 \quad\text{in $B_M^c(x_0)$.}
$$
where $U = B_M(x_0)\setminus \overline{B_{m}(x_0)}$. Then the comparison principle yields $\beta\geq u_{\lambda}$ in $\R^d$, whence, if $\phi\in \partial B_M(x_0)\cap \partial \Omega_{\lambda}$, we get
$$
D_{U}^{1/2} \beta(\phi)\geq  D_{\Omega_{\lambda}}^{1/2} u_{\lambda}(\phi) = \lambda\,.
$$
Using again Lemma \ref{boundary calc}, we also obtain
$$
D_{U}^{1/2} \beta(\phi)\leq \frac{1}{\sqrt{M-m}}.
$$
Since $M-m\geq R-r_K=\dist(\partial \Omega_{\lambda},\partial K)$, putting together the last four inequalities, we obtain
$$
\frac{C}{\sqrt{\dist(\partial \Omega_{\lambda},\partial K)}}\left(\frac{r_K}{r_k+\dist(\partial \Omega_{\lambda},\partial K)}\right)^{d-\frac12}\leq D_{W}^{1/2} b(\theta)\leq \lambda\leq 
D_{U}^{1/2} \beta(\phi)\leq \frac{1}{\sqrt{\dist(\partial \Omega_{\lambda},\partial K)}}.
$$
Thus
\begin{equation}\label{bound on distance}
	C^2 \left(\frac{r_K}{r_K+\dist(\partial \Omega_{\lambda},\partial K)}\right)^{2d-1}\frac{1}{\lambda^2}\leq \dist(\partial \Omega_{\lambda},\partial K)\leq \frac{1}{\lambda^2}.
\end{equation}
The right inequality in (iv) is proved. Moreover, we note that sending $\lambda\to\infty$, this implies $\dist(\partial \Omega_{\lambda},\partial K)\to 0$, so also the first assertion in  (iii) is proved.\\
Next, note that the first inequality on \eqref{bound on distance} gives
\begin{equation}
	\label{d2}\dist(\partial \Omega_{\lambda},\partial K)^{1/(2d-1)}\big(r_K+\dist(\partial \Omega_{\lambda},\partial K)\big)\ge\frac{C^{2/(2d-1)} r_K}{\lambda^{2/(2d-1)}}\,.
\end{equation}
This shows the second assertion of (iii).\\
(iv): Now, set 
$$
t=\frac{\dist(\partial \Omega_{\lambda},\partial K)^{1/(2d-1)}}{r_K^{1/(2d-1)}}\,,
$$
then \eqref{d2} reads
\begin{equation}\label{h0}
	h(t):=t^{2d}+t-A\ge0\,,
\end{equation}
where
$$
A=\left(\frac{C^2}{r_K\lambda^{2}}\right)^{1/(2d-1)}\,.
$$
Notice that $h$ is $C^\infty(\R)$,  $h(0)<0$ and $\lim_{t\to+\infty}h(t)=+\infty$. Since $h$ is strictly increasing for $t\geq0$, we have that there exists exactly one $t(A)>0$ such that $h(t(A))=0$ and \eqref{h0} is equivalent to $t\geq t(A)$ . Put
\begin{equation}\label{gt}
	g_{d,r_K}(\lambda)=r_K{t}(A)^{2d-1}\,.
\end{equation}
To obtain estimate \eqref{gdrk}, observe that
$$
h\left(A\right)=A^{2d}>0\quad\text{and}\quad h\left(A^{1/2d}\right)=A^{1/2d}>0\,,
$$
hence 
$$
t(A)<\min\{A,A^{1/2d}\}\,,
$$
and there exists $\epsilon<1$ such that
$$
t(A)=\epsilon\min\{A,A^{1/2d}\}\,.
$$
Let's look for an estimate from below of $\epsilon$ and calculate
\begin{align*}
h(\epsilon\min\{A,A^{1/2d}\})&=\epsilon^{2d}\min\{A,A^{1/2d}\}^{2d}+\epsilon\min\{A,A^{1/2d}\}-A\\
&<\epsilon(\min\{A,A^{1/2d}\}^{2d}+\min\{A,A^{1/2d}\})-A\,,
\end{align*}
and we see that, choosing
$$
\epsilon=\frac{A}{\min\{A,A^{1/2d}\}(\min\{A,A^{1/2d}\}^{2d-1}+1)}\,,
$$
we have $h(\epsilon\min\{A,A^{1/2d}\})<0$, which yields
$$
t(A)>\frac{A}{\min\{A,A^{1/2d}\}^{2d-1}+1}\,.
$$
The estimate \eqref{gdrk} follows from the latter coupled with \eqref{gt}. 
\end{proof}

Now our aim is to show Theorem \ref{inward_normal_ray}. We will use some ideas from \cite{S1971} and \cite{S1994}. First we need some auxiliary lemmas.  {For this recall the notions of halfspaces $H_{\lambda,e}$ and reflections $Q_{\lambda,e}$ introduced in the beginning of Section \ref{preliminearies}.}

\begin{lemma}\label{psi_lemma}
	Assume $\Omega\subset \R^d$ is an open set with $C^2$ boundary and let $H=\{x\in \R^d\;:\; x_1>0\}$. Assume $0\in \partial \Omega$ is such that $e_2=(0,1,0,\ldots,0)$ is the interior normal of $\Omega$ at $0$. Let $\delta(x)=\dist(x,\R^d\setminus \Omega)$, $\eps > 0$ and $u$ be a function belonging to $C^s(B_{\eps}(0))$ for some $s\in(0,1)$ and satisfying
	\begin{enumerate}[(i)]
		\item $u(x)=0$ for $x \in B_{\eps}(0) \setminus \Omega$,
		\item $u(x)= \delta^{s}(x) \psi(x)$ for $x \in \overline{\Omega} \cap B_{\eps}(0)$, where $\psi \in C^{1}(\overline{\Omega \cap B_{\eps}(0)})$ 
		\item $\psi\equiv c$ on $\partial \Omega \cap B_{\eps}(0)$.
	\end{enumerate}
	Let $\eta=(1,1,0,\ldots,0)$. Then $v:=\bar{u}-u=u\circ Q_{0,e_1}-u$ satisfies
	\begin{equation}
		v(t\eta)=o(t^{1+s})\quad\text{as $t\to 0^+$.}
	\end{equation}
\end{lemma}
\begin{proof}
First of all,  notice that $\nabla \delta(0)=e_2$. We assume $t$ is sufficiently small such that $t\eta$ and $\overline{t\eta}=t\overline{\eta}$ belong to $\Omega$. Notice next that we have
\begin{eqnarray}
v(t\eta)&=&	u(t \overline{\eta}) - u(t\eta)\notag\\
 &=&  \delta^s(t \overline{\eta}) \psi(t\overline{\eta}) - \delta^s(t \eta) \psi(t\eta)\notag\\
	&=& (\delta^s(t \overline{\eta}) - \delta^s(t \eta)) \psi(t\overline{\eta}) +
	\delta^s(t \eta) (\psi(t\overline{\eta}) - \psi(t\eta))\label{small o estimate}
\end{eqnarray}	
We begin with the first summand. It holds by a Taylor expansion for $t\to 0^+$
\begin{eqnarray*}
\delta(t\eta)&=&\delta(0)+\nabla \delta(0)t\eta +\frac12 \nabla^2\delta(0)[t\eta](t\eta)+o(t^2)\\
&=& t+\frac{t^2}2\Big(\nabla^2\delta(0)[e_1](e_1)+2\nabla^2\delta(0)[e_1](e_2)+\nabla^2\delta(0)[e_2](e_2)\Big)+o(t^2),
\end{eqnarray*}	
where, since $\nabla \delta(0)\cdot e_1=e_2\cdot e_1=0$, 
$$
|\nabla^2\delta(0)[e_1](e_2)|= \lim_{t\to0}\frac{|\nabla \delta(te_2)\cdot e_1|}{t} \leq \lim_{t\to0}\max_{\tau\in[0,t]}|\nabla^2 \delta(\tau e_2)[e_1]|=0
$$
using $\delta\in C^2(\overline{\Omega})$. Thus
\begin{eqnarray*}
\delta(t\eta)&=& t+\frac{t^2}2\Big(\nabla^2\delta(0)[e_1](e_1)+\nabla^2\delta(0)[e_2](e_2)\Big)+o(t^2)
\end{eqnarray*}
 and similarly,
\begin{eqnarray*}
	\delta(t\overline{\eta})&=&t+\frac{t^2}2\Big(\nabla^2\delta(0)[e_1](e_1)+\nabla^2\delta(0)[e_2](e_2)\Big)+o(t^2).
\end{eqnarray*}	
Thus, for $t>0$, there is $\tau\in [0,1]$ such that with $x_t=\eta +\tau(\overline{\eta}-\eta)=(1-2\tau)e_1+e_2$ it holds for $t\to 0^+$
\begin{eqnarray*}
\delta^s(t \overline{\eta}) - \delta^s(t \eta)&=& s\delta^{s-1}(tx_{t})\Big(\delta(t\overline{\eta})-\delta(t\eta)\Big)\\
&=& s\delta^{s-1}(tx_{t})o(t^2)\\
&=& o(t^{1+s}),
\end{eqnarray*}
noting that $tx_t\to 0$ for $t\to0$ and that $\delta(tx_t)$ is comparable to $t$. Since $\psi(t\overline{\eta})=c+o(1)$ as $t\to 0^+$ by assumption, we have
$$
(\delta^s(t \overline{\eta}) - \delta^s(t \eta)) \psi(t\overline{\eta})=o(t^{1+s})\quad\text{as $t\to 0^+$}
$$	
which shows that the first summand in \eqref{small o estimate} behaves as claimed.\\
For the second summand in \eqref{small o estimate}, note first that there is $C>0$ such that
$$
C^{-1}t^{s}\leq \delta^s(t \eta)\leq C t^{s}\quad\text{for $t\geq 0$ small enough.}
$$
It thus remains to show that
\begin{equation}\label{estimate to be shown}
	\psi(t\overline{\eta}) - \psi(t\eta)=o(t)\quad\text{as $t\to 0^+$}.
\end{equation}
Similarly as above, there is, for every $t>0$ some $\tau\in [0,1]$ such that with $x_t=\eta +\tau(\overline{\eta}-\eta)=(1-2\tau)e_1+e_2$ it holds
\begin{eqnarray}
\psi(t\overline{\eta}) - \psi(t\eta)&=& t\nabla \psi(tx_t) \cdot (\overline{\eta}-\eta)\notag\\
	&=& -2t\partial_1\psi(tx_t).\label{estimate to be shown2}
\end{eqnarray}
In the following, let $\tilde{\psi}:\R^d\to\R$ be a function in $C^1(B_{\eps}(0))$ such that $\tilde{\psi}=\psi$ and $\nabla \tilde{\psi}=\nabla\psi$ in $\overline{\Omega} \cap B_{\eps}(0)$ ---this is possible by Whitney's theorem, see \cite[Theorem]{W34} or \cite[Section 2.5]{BB12}, using that $\partial\Omega$ is of class $C^2$. For $r\in \R$ and $f:U\to \R$, where $U\subset \R^d$ is an arbitrary set, we write
$$
L_r(f):=\{x\in U\;:\; f(x)=r\}.
$$
Note that since $\psi\equiv c$ on $\partial \Omega \cap B_{\eps}(0)$ we have $0\in \partial\Omega\subset L_c(\psi)\subset L_c(\tilde{\psi})$. Since $\tilde{\psi}$ is a $C^1$ function, it holds $\nabla\psi(0)=\nabla\tilde{\psi}(0)=0$ or $\nabla \tilde{\psi}=\nabla\psi$ is orthogonal to the tangent plane at $0$. In the first case, we immediately have $\partial_1\psi(0)=0$ and in the second case, we note that $e_1$ is contained in the tangent plane and thus, again, $\partial_1\psi(0)=\nabla \psi(0)e_1=0$. Thus, it follows
$$
\partial_1\psi(tx_t)=o(1)\quad\text{for $t\to 0^+$,}
$$
and \eqref{estimate to be shown} follows with \eqref{estimate to be shown2}
\end{proof}

For any halfspace $H \subset \R^d$ and any $x \in \R^d$ we denote by
$$
\hat{x}^{\partial H}
$$
the reflection of $x$ with respect to $\partial H$.

\begin{rem}
Lemma \ref{psi_lemma} is very similar to \cite[Lemma 4.3]{FJ2015}. However, in \cite[Lemma 4.3]{FJ2015} the missing assumption that 
\begin{equation}\label{additional missing assumption}
u/\delta^s:\Omega\to\R \quad\text{ has a $C^1$ extension to a function defined on $\overline{\Omega}$}  
\end{equation}
 is necessary to conclude the result. In particular, the main results, Theorem 1.1 and Theorem 1.2 in \cite{FJ2015} need the additional assumption \eqref{additional missing assumption}.  We emphasize that even in the classical overdetermined problem by Serrin \cite{S1971}, that is the case $s=1$, the solution is assumed to be in $C^2(\overline{\Omega})$. Note that the assumption \eqref{additional missing assumption} and a similar remark appear also in \cite{DPTV23}. A slightly adjusted assumption \eqref{additional missing assumption} is also needed in subsequent results to \cite{FJ2015}, for instance, in \cite{SV19}.
\end{rem}

\begin{lemma}
	\label{normal_derivative}
	Let $H=\{x\in \R^d\;:\; x_1>0\}$ and $\hat{H} = \R^d \setminus \overline{H}=\{x\in \R^d\; :\; x_1<0\}$. Let $\Omega \subset \R^d$ be an open, nonempty, connected, bounded set, which has a $C^{2,1/2}$ boundary, $K \subset \Omega \cap \hat{H}$ be an open, nonempty, and bounded set {{which has a $C^{2}$ boundary,}} such that $\dist(K, \partial(\Omega \cap \hat{H})) > 0$. Put $\Omega_- = \Omega \cap \hat{H}$, $\Omega_+ = \Omega \cap H$. We assume that $Q_{0,e_1}(\Omega_+) \subset \Omega_-$ and that there exists $z \in \partial \Omega \cap \partial H$ such that $\partial \Omega$ and $\partial H$ are perpendicular at $z$ (see Figure \ref{fig:normal_derivative_picture_2}). Assume that $u: \R^d \to [0,1]$, {{$u \in H^{1/2}(\R^d)$}},  $u = 1$ on $\overline{K}$, $u = 0$ on $\Omega^c$, $(-\Delta)^{1/2} u = 0$ on $\Omega \setminus \overline{K}$. Then $\Dn  u$ is not constant on $\partial \Omega$.
\end{lemma}
\begin{proof}

%

\begin{figure}
\centering
\begin{tikzpicture}[x=1cm,y=1cm,step=1cm]
\node at (0,0) {\includegraphics[width=0.5\textwidth]{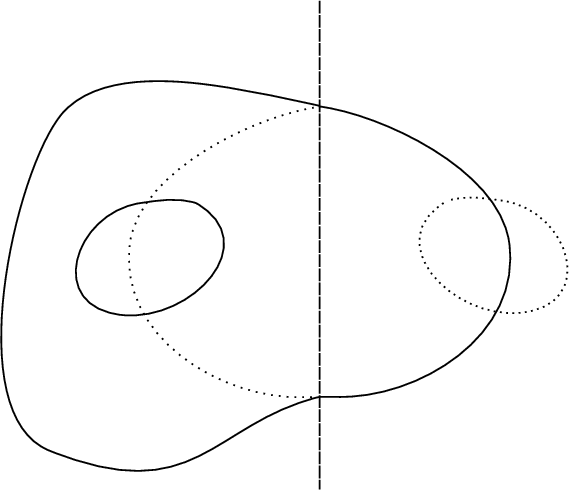}};
\node at (1.5,0.5) {$\Omega_+$};
\node at (-0.3,0.5) {$\hat{\Omega}_+$};
\node at (-2,0) {$K$};
\node at (3,0) {$\hat{K}$};
\node at (-2.8,1.2) {$\Omega_-$};
\node at (-0.1,3) {$\hat{H}$};
\node at (1,3) {$H$};
\fill[black] (0.52,-2.27) circle[radius=0.07]; 
\node at (0.7,-2.6) {$0$};
\draw[->,dashed]  (-5,-2.27) -- (5,-2.27);
\node at (4.8,-2.6) {$x_1$};
\end{tikzpicture}
\caption{Picture for Lemma \ref{normal_derivative}.} 
\label{fig:normal_derivative_picture_2}
\end{figure}
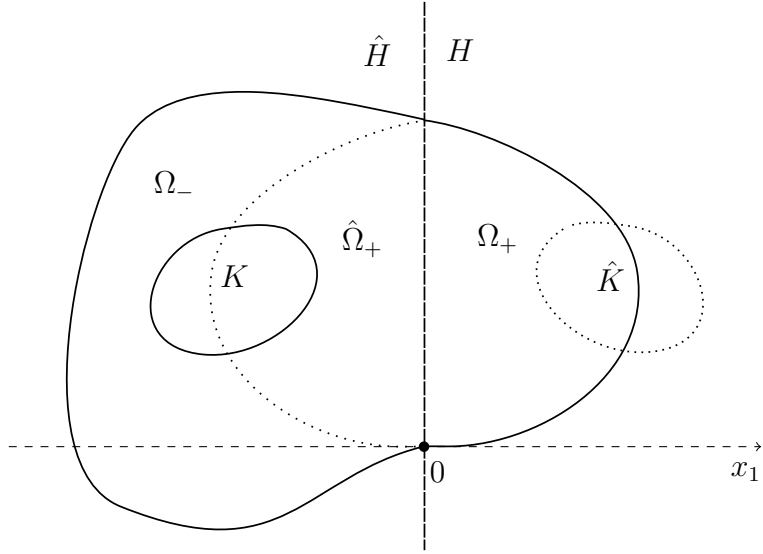

{{The fact that $\Dn  u$ is well defined on $\partial \Omega$ follows from Lemma \ref{C12}.}}
We may assume that $z = 0$.
By contradiction, assume that $\Dn  u$ is constant on $\partial \Omega$, i.e. $\Dn  u(x) = c$ for all $x \in \partial \Omega$ for some $c>0$. Let $\eps > 0$ be such that $\dist(B_{\eps}(0),K) > 0$. Then $u$ clearly satisfies
\begin{equation*}
\left\{
\begin{aligned}
(-\Delta)^{1/2} u(x)&=0 &&\text{for $x \in \Omega \cap B_{\eps}(0)$,}\\
u(x)&=0 &&\text{for $x \in B_{\eps}(0) \setminus \Omega$.}
\end{aligned}
\right.
\end{equation*}
Let $\delta(x)=\dist(x,\R^d\setminus \Omega)$. By \cite[Theorem 1.4]{AR2020} for $x \in \overline{\Omega} \cap B_{\eps}(0)$ we have the representation  $u(x) = \delta^{1/2}(x) \psi(x)$, where $\psi \in C^{1,1/2}(\overline{\Omega} \cap B_{\eps}(0))$. By our assumption that for all $x \in \partial \Omega$ we have $\Dn  u(x) = c$ we obtain that for all $x \in \partial \Omega \cap B_{\eps}(0)$ we have $\psi(x) = c$. 
Put $\eta = e_1 + e_2 = (1,1,0,\ldots,0)$ and $\overline{\eta} = -e_1 + e_2 =  (-1,1,0,\ldots,0)$. By Lemma \ref{psi_lemma} we have
\begin{equation}\label{tocontradict}
u(t\overline{\eta}) - u(t\eta) = o(t^{3/2})\quad\text{as $t\to 0^+$.}
\end{equation}

For $x  \in \R^d$ let $\overline{x} = Q_{0,e_1}(x)$ and put
$$
v(x) = u(\bar{x})-u(x).
$$

Note that $(-\Delta)^{1/2} u(x) = 0$ for $x \in \Omega \setminus \overline{K}$ so $(-\Delta)^{1/2} v(x) = 0$ for $x \in \Omega_+ \setminus \hat{K}$, where $\hat{K}:=Q_{0,e_1}(\overline{K})$. Since $\hat{K} \subset H$ we have $v(x) > 0$ for $x \in \hat{K}$. We also have $v(x) \ge 0$ for $x \in H \setminus \Omega_+$. Hence, by Proposition \ref{prop:hopf2} for $W = \Omega_+ \setminus \hat{K}$, we have $v(x) > 0$ for $x \in \Omega_+ \setminus \hat{K}$.
It is easy to show that there exists an open set $U \subset \Omega \setminus \overline{K}$ such that $0 \in \partial U$ with $C^2$ boundary which is symmetric across $\partial H$. We may assume that $\dist(U,\overline{K}) > 0$. By Lemma \ref{cornerpoint} there exist positive constants $c_0$ and $t_0$ such that for all $t \in (0,t_0)$ we have
\begin{equation*}
v(t\overline{\eta}) = u(t\overline{\eta}) - u(t\eta) \ge c_0 t^{3/2}.
\end{equation*}
This contradicts \eqref{tocontradict}, so ending the proof of the lemma.
\end{proof}

\begin{lemma}\label{moving plane argument}
Assume that a bounded domain $K \subset \R^d$ has a $C^2$ boundary and it is starshaped with respect to some open ball contained in $K$ and let $u=u_{\lambda}$, $\Omega=\Omega_{\lambda}$ be the solution of Problem \ref{bernoulli_problem} given by Theorem \ref{Jarohs_Kulczycki_Salani} for some $\lambda>0$. Let $e\in \partial B_1(0)$ and, for $t\in \R$ let
$$
v_{t}:=u\circ Q_{t,e}-u,
$$
using the notation as in the beginning of Section \ref{preliminearies}. Define $\Omega_t=\Omega\cap H_{t,e}$, $K_t=K\cap H_{t,e}$, and
\begin{align*}
t_1&:=t_1(e):=\sup\{t \in \R\;:\; H_{t,e}\cap \Omega\neq \emptyset\}\quad\text{and}\\
 t_0&:=t_0(e):=\inf\{t\in \R\;:\; Q_{\mu,e_1}(\Omega_{\mu})\subset \Omega\text{ and } Q_{\mu,e_1}(K_{\mu})\subset K\text{ for all $\mu\in(t,t_1)$}\}
\end{align*}
Then the following are true. 
\begin{enumerate}
\item[(i)] It holds $t_0 < t_1<\infty$ and for any $t\in(t_0,t_1)$ we have
\begin{itemize}
\item $v_{t}\in C^{1/2}(\R^d)\cap C^{\infty}(\Omega_{t}\setminus \overline{M_{t}})$, where $M_t:=Q_{t,e}(K\setminus \overline{H_{t,e}})$,
\item $(-\Delta)^{1/2}v_{t}=0$ in $\Omega_{t}\setminus \overline{M_{t}}$, $v_{t}\geq 0$ in $H_{t,e}\setminus (\Omega_{t}\setminus \overline{M_{t}})$, and $v_{t}\circ Q_{t,e_1}=-v_{t}$ in $\R^d$.
\end{itemize}
\item[(ii)] $v_t>0$ in $\Omega_{t}\setminus \overline{M_{t}}$ for any $t\in(t_0,t_1)$ and 
$$
\text{either $v_{t_0}\equiv 0$ in $\R^d$ \quad  or $\quad v_{t_0}>0$ in $\Omega_{t_0}\setminus \overline{M_{t_0}}$.}
$$
\end{enumerate}
\end{lemma}
\begin{proof}
Recall that $u$ belongs to the energy space $H^{1/2}(\R^d)$ as mentioned above. Note that by regularity theory, it follows that $u=u_{\lambda}\in C^{1/2}(\R^d)$, see e.g. \cite[Theorem 1.4]{AR2020} for the regularity at $\partial \Omega$ and the regularity at $\partial K$ follows since this boundary is of class $C^2$ and by considering the problem solved by $1-u$ (in a localized sense).\\
We show the claim with the moving plane method. For this, let $e\in \partial B_1(0)$ and consider, for $t\in \R$ the hyperplane $H_{t}=\{x\in \R^d\;:\; x\cdot e>t\}$. By rotation we may assume that $e=e_1=(1,0,\ldots,0)$. We emphasize the following four situations that may occur for some $t$:
\begin{enumerate}
 	\item The boundary of $Q_{t,e_1}(\Omega_{t})$ touches the boundary of $\Omega$ in $\{x\in \R^d\;:\; x_1<t\}$ at some point $z$;
 	\item $\partial H_{t}$ is perpendicular to $\partial \Omega$ at some point $z \in \partial \Omega$;
 	\item The boundary of $Q_{t,e_1}(K_{t})$ touches the boundary of $K$ in $\{x\in \R^d\;:\; x_1<t\}$ at some point $z$;
 	\item $\partial H_{t}$ is perpendicular to $\partial K$ at some point $z \in \partial K$.
 \end{enumerate} 
 For (i), first note that $t_0 < t_1<\infty$ follows immediately by Theorem \ref{Jarohs_Kulczycki_Salani}. Next note that at $t_0$ at least one of the above cases 1.--4. is true.
 Notice that $K_{t}\subset M_{t}$ for $t\geq t_{0}$ so that $v_{t}\geq 0$ in $\overline{M_{t}}$. Moreover, $v_{t}\geq 0$ in $H_{t}\setminus \Omega_{t}$ since $u\geq 0$ in $\R^d$. Finally, notice that $v_{t}\in C^{1/2}(\R^d)\cap C^{\infty}(\Omega_{t}\setminus \overline{M_{t}})$ and this function satisfies
$$
(-\Delta)^{1/2}v_{t}=0\quad\text{in $\Omega_{t}\setminus \overline{M_{t}}$,}\quad v_{t}\geq 0\quad\text{in $H_{t}\setminus (\Omega_{t}\setminus \overline{M_{t}}),\ $ and}\quad v_{t}\circ Q_{t,e_1}=-v_{t}\quad\text{in $\R^d$}
$$
for any $t\in[t_0,t_1)$. This shows (i).\\
For (ii), we stress that $v_{t}\equiv0$ is not possible for any $t>t_0$, since $u>0$ in $\Omega$ and $u=1$ in $K$. Proposition \ref{prop:hopf2} thus entails that for all $t\in(t_0,t_1)$ we have $v_{t}>0$ in $\Omega_{t}\setminus \overline{M_{t}}$ and we have either
\begin{equation}\label{eq:moving plane argument}
v_{t_0}\equiv 0\quad \text{in $\R^d$} \quad \text{or}\quad v_{t_0}>0\quad\text{in $\Omega_{t}\setminus \overline{M_{t}}$.}
\end{equation}
This finishes the proof.
\end{proof}

\begin{proof}[Proof of Theorem \ref{inward_normal_ray}]
Fix $\lambda > 0$ and abbreviate $\Omega = \Omega_{\lambda}$ and $u=u_{\lambda}$. Recall that $u$ belongs to the energy space $H^{1/2}(\R^d)\cap C^{1/2}(\R^d)$ as mentioned in the proof of Lemma \ref{moving plane argument}.

By contradiction, assume that there is $x^\ast\in \partial \Omega$ such that the inward normal ray to $\partial \Omega$ at $x^\ast$ does not meet $\textrm{conv}(\overline{K})$ (see Figure \ref{fig:contradiction}).

Then, there exists a $(d-1)$-dimensional hyperplane $T^\ast$ containing this inward normal ray such that $\textrm{conv}(\overline{K}) \cap T^\ast = \emptyset$. 
Clearly, $T^\ast$ is orthogonal to $\partial \Omega$ at $x^\ast$.

\begin{figure}
\centering
\begin{tikzpicture}[x=1cm,y=1cm,step=1cm]
\node at (0,0) {\includegraphics[width=0.5\textwidth]{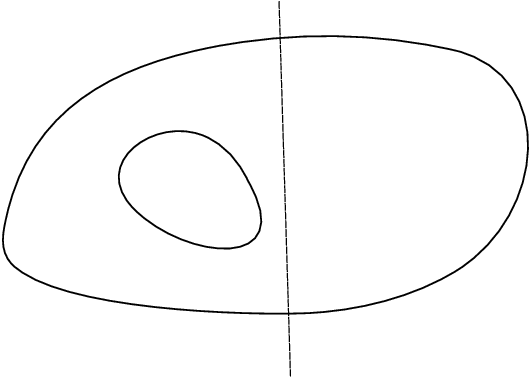}};
\node at (-1,1.5) {$\Omega$};
\node at (1.5,0.5) {$\Omega_{t^{\ast}}$};
\node at (1.3,-2.2) {$x^{\ast}\in \partial \Omega$};
\node at (-1.8,0) {$K$};
\node at (0.65,2.8) {$T^{\ast}$};
\fill[black] (0.385,-2.01) circle[radius=0.07]; 
\draw[line width=0.4pt] (0.781,-2.01) arc (2:87:0.44);
\fill[black] (0.56,-1.83) circle[radius=0.03]; 
\end{tikzpicture}
\caption{What happens when the inward normal ray to $\partial \Omega$ at $x_0$ does not meet $\overline{K}$.} 
\label{fig:contradiction}
\end{figure}
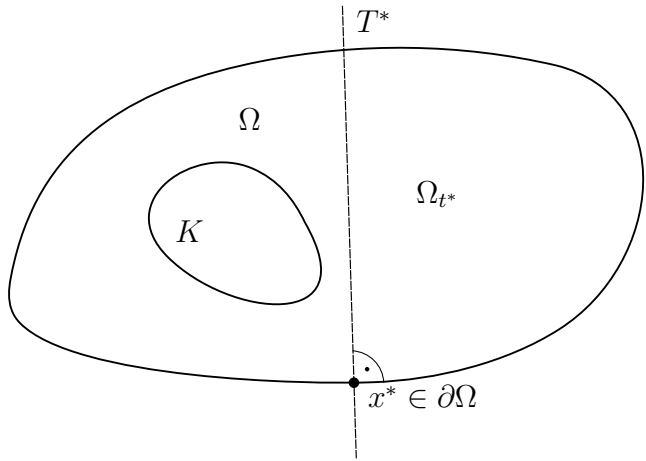

By translation and rotation, we can assume $x^\ast=0$, $T^\ast=\{x=(x_1,\dots,x_d)\in \R^d\;:\; x_1=0\}$, and for all $x\in \textrm{conv}(\overline{K})$ it holds $x_1<0$.

Next, for $t\in \R$ let $H_t=\{x\in \R^d\;:\; x_1>t\}$ and let $v_t=u\circ Q_{t,e_1}-u$. 
Moreover, let $t_0,t_1$ be as stated in Lemma \ref{moving plane argument} and let
$t^\ast$ be the first $t$ such that one of the cases 1.--4. in the proof of Lemma \ref{moving plane argument} happens. Then, $t^\ast\geq 0$ by our assumption and $t^\ast\geq t_0$ by definition.  If $t^\ast>t_0$, then $v_{t^\ast}>0$ in $\Omega_{t^\ast}\setminus \overline{M_{t^\ast}}$
by (ii) of Lemma \ref{moving plane argument}. If $t^\ast=t_0$, then \eqref{eq:moving plane argument} holds. But $v_{t^\ast}\equiv 0$ is not possible, since $K\subset (H_{t^\ast})^c$ and $u=1$ in $K$, while $u<1$ in $(\overline{K})^c$. This implies again $v_{t^\ast}>0$ in $\Omega_{t^\ast}\setminus \overline{M_{t^\ast}}$.

Now, notice that our assumption implies $t^\ast\geq 0$ and that one of the cases 1. or 2. is happening at $t^\ast$.

If we are in case 1: Then there is an internal ball $B\subset \Omega_{t^\ast}\setminus \overline{M_{t^\ast}}$ such that $\overline{B}\cap \partial \Omega=\{z\}$. Proposition \ref{prop:hopf2} applied to $v_{t^\ast}$ in $W=\Omega_{t^\ast}\setminus \overline{M_{t^\ast}}$ implies 
$$
0< D_{\Omega_{t^\ast}}^{1/2} v_{t^\ast}(z)= D_{Q_{t^\ast,e_1}(\Omega_{t^\ast})}^{1/2} u(Q_{t^\ast,e_1}(z))- D_{\Omega_{t^\ast}}^{1/2}u(z)=\lambda-\lambda=0\,,
$$
a contradiction!

If we are in case 2: Then, by translation if necessary, we may assume $t^\ast=0$ and apply Lemma \ref{normal_derivative} to $u$. But this leads again to a contradiction, since $\Dn u$ is constant on $\partial \Omega$.\\
This finishes the proof.
\end{proof}
\medskip

\noindent{\large{\bf{\em Acknowledgements}}}
The authors would like to thank an unknown referee for instructive comments which led to a substantial improvement of the paper.
The second author  was supported by the National Science Centre, Poland, grant no. 
2019/33/B/ST1/02494.
The third author was supported in part by INdAM through a GNAMPA Project and by 
the project "Geometric-Analytic Methods for PDEs and Applications (GAMPA)", funded by European Union --Next Generation EU  within the PRIN 2022 program 
(D.D. 104 - 02/02/2022 Ministero dell'Universit\`a e della Ricerca).

\end{document}